\documentclass{amsproc}
\usepackage{amssymb}
\usepackage{amsfonts}

\theoremstyle{plain}

\newtheorem{corollary}{Corollary}

\newtheorem{definition}{Definition}

\newtheorem{lemma}{Lemma}

\newtheorem{proposition}{Proposition}
\newtheorem{remark}{Remark}

\newtheorem{theorem}{Theorem}
\numberwithin{equation}{section}
\input{tcilatex}

\begin{document}
\title[Union property]{The Ubiquity of Sidon sets that are not $I_{0}$ }
\author{Kathryn E. Hare}
\address{Dept. of Pure Mathematics\\
University of Waterloo\\
Waterloo, Ontario\\
Canada N2L 3G1}
\email{kehare@uwaterloo.ca}
\thanks{This research is supported in part by NSERC \#44597}
\author{L. Thomas Ramsey}
\address{Dept. of Mathematics\\
University of Hawaii\\
Honolulu, HI., USA 96822}
\email{ramsey@math.hawaii.edu}
\thanks{This paper is in final form and no version of it will be submitted
for publication elsewhere.}

\begin{abstract}
We prove that every infinite, discrete abelian group admits a pair of $I_{0}$
sets whose union is not $I_{0}$. In particular, this implies that every such
group contains a Sidon set that is not $I_{0}$.
\end{abstract}

\maketitle

\section{Introduction}

A subset $E$ of a discrete abelian group $\Gamma $ with compact dual group $%
G $ is said to be a Sidon set if every bounded $E$-sequence can be
interpolated by the Fourier transform of a measure on $G$. If the measure
can be chosen to be discrete, $E$ is called an $I_{0}$ set. Finite sets in
any group $\Gamma $ and Hadamard sequences of integers are examples of these
interpolation sets in the group $\Gamma =\mathbb{Z}$.

Clearly, $I_{0}$ sets are Sidon. However, the converse is not true since the
class of Sidon sets is known to be closed under finite unions, but the class
of $I_{0}$ sets is not. Indeed, in \cite{Me}, M\'{e}la gave an example of a
pair of Hadamard sets in $\mathbb{Z}$ whose union is not $I_{0}$.

In this note we prove that every infinite, discrete abelian group admits a
pair of $I_{0}$ sets whose union is not $I_{0}$. Consequently, every such
group admits a Sidon set that is not $I_{0}$. Our method is constructive and
establishes even more: We prove that given any infinite subset $F\subseteq
\Gamma $ there are $I_{0}$ sets $E\subseteq F$ and $E^{\prime }\subseteq
F+F-F,$ whose union is not $I_{0}$ (but is, of course, Sidon). In fact, we
show that the sets $E,$ $E^{\prime }$ have stronger interpolation properties
than just $I_{0}$. These depend upon the algebraic properties of the initial
set $F$.

\section{Preliminaries}

Let $G$ be a compact abelian group and $\Gamma $ its discrete abelian dual
group. One example is $G=$ $\mathbb{T}$, the circle group of complex numbers
of modulus one, with dual group $\Gamma =$ $\mathbb{Z}$.

\begin{definition}
(i) A subset $E\subseteq \Gamma $ is said to be \textbf{Sidon} if for every
bounded function $\phi :E\rightarrow \mathbb{C}$ there is a measure $\mu $
on $G$ with $\widehat{\mu }(\gamma )=\phi (\gamma )$ for all $\gamma \in E$.
If the interpolating measure $\mu $ can always be chosen to be discrete,
then the set $E$ is said to be\textbf{\ }$I_{0}$.

(ii) A subset $E\subseteq \Gamma $ is said to be $\varepsilon $\textbf{%
-Kronecker} if for every $\phi :E\rightarrow \mathbb{T}$ there exists $x\in
G $ such that 
\begin{equation}
\left\vert \phi (\gamma )-\gamma (x)\right\vert <\varepsilon \text{ for all }%
\gamma \in E  \label{Kronecker}
\end{equation}%
and is called \textbf{weak }$\varepsilon $\textbf{-Kronecker }if the strict
inequality is replaced by $\leq $.
\end{definition}

Hadamard sets $E=\{n_{j}\}\subseteq \mathbb{N}$ with Hadamard ratio $q=\inf
n_{j+1}/n_{j}>1$ are $I_{0}$ sets \cite{KR}, thus both $\{3^{j}\}_{j\geq 1}$
and $\{3^{j}+j\}_{j\geq 1}$ are $I_{0}$ subsets of $\mathbb{Z}$. However, M%
\'{e}la, in \cite{Me}, proved that their union is not $I_{0}$. This shows
that not all Sidon sets in $\mathbb{Z}$ are $I_{0}$ since it is a deep
result, first obtained by Drury \cite{Dr}, with a later proof given by
Pisier \cite{Pi}, that any finite union of Sidon sets (in any group $\Gamma
) $ is Sidon. It is an open problem whether every Sidon set is a finite
union of $I_{0}$ sets.

Hadarmard sets with ratio $q>2$ are also known to be weak $\varepsilon $%
-Kronecker with $\varepsilon =$ $\left\vert 1-e^{i\pi (q-1)}\right\vert $.
In fact, every $\varepsilon $-Kronecker set is $I_{0}$ if $\varepsilon <%
\sqrt{2}$ \cite{GH1} and is Sidon if $\varepsilon <2$ \cite{HR}.

It is known that every infinite subset of $\Gamma $ contains infinite $I_{0}$
sets and if $\Gamma $ does not contain any elements of order 2, then every
infinite subset contains an infinite weak $1$-Kronecker set; c.f., \cite{GaH}%
, \cite{GHColloq}, \cite{GH13} and the references cited there. Of course, if 
$\Gamma $ consists of only elements of order two, then it does not contain
any $\varepsilon $-Kronecker set with $\varepsilon <\sqrt{2}$.

The main result of this paper is to show that \textit{every} infinite,
discrete abelian group admits a pair of $I_{0}$ sets that are weak $%
\varepsilon $-Kronecker for suitable $\varepsilon $, but whose union is not $%
I_{0}$. The number $\varepsilon $ can always be chosen to be at most $\sqrt{2%
}$ and often can be taken to be arbitrarily small, such as when $\Gamma $ is
torsion free. As far as we are aware, this is the first proof that every
infinite, discrete abelian group admits Sidon sets that are not $I_{0}$.
Knowing the existence of such sets is useful, for instance, in studying the
space of weakly almost periodic functions on $\Gamma $; c.f., \cite[p. 17]%
{FG}, as well as \cite{FGAdv} for further background.

Like M\'{e}la's original proof, our argument is constructive and relies upon
the well known Hartman/Ryll-Nardzewski characterization of $I_{0}$ sets in
terms of the Bohr topology on $\overline{\Gamma }$, the Bohr
compactification of $\Gamma $.

\begin{proposition}
\cite{HRN}\label{P:HRN} A subset $E\subseteq \Gamma $ is $I_{0}$ if and only
if whenever $E_{1}$ and $E_{2}$ are disjoint subsets of $E$, then $E_{1}$
and $E_{2}$ have disjoint closures in $\overline{\Gamma }$.
\end{proposition}

With this, one can quickly prove the following generalization of M\'{e}la's
argument.

\begin{lemma}
\label{Mela}Suppose that $0$ is a cluster point of $\{\chi
_{n}\}_{n=1}^{\infty }\subseteq \Gamma $ in the Bohr compactification of $%
\Gamma $. If the sets $E=\{\gamma _{n}\}\subseteq \Gamma $ and $E^{\prime
}=\{\gamma _{n}+\chi _{n}\}\subseteq \Gamma $ are disjoint, then $E\cup
E^{\prime }$ is not an $I_{0}$ set.
\end{lemma}

\begin{proof}
Assume the subnet $\{\chi _{n_{\alpha }}\}$ (indexed by $\alpha $) converges
to $0$ in $\overline{\Gamma }$. Since $\overline{\Gamma }$ is compact, there
is a subnet (not relabelled) such that $\{\gamma _{n_{\alpha }}\}$ converges
to some $\gamma \in \overline{\Gamma }$. Since \{$\chi _{n_{\alpha }}\}$
also converges to $0$ along this subnet, it follows that $\{\gamma
_{n_{\alpha }}+\chi _{n_{\alpha }}\}$ converges to $\gamma $. But that means
the disjoint sets $E$ and $E^{\prime }$ both have $\gamma $ in their
closures. By Proposition \ref{P:HRN}, their union is not $I_{0}$.
\end{proof}

In the case of $\Gamma =\mathbb{Z}$, one can take $\{\chi _{n}\}=\mathbb{N}$%
. For the more general situation, we will take $\{\chi
_{n}\}=\tbigcup\limits_{m=1}^{\infty }H_{m}$, where the sets $H_{m}$ are
constructed in the following lemma.

\begin{lemma}
\label{clustering} If $F$ is any infinite subset of $\Gamma ,$ then there is
a countable subset $H\subseteq (F-F)\diagdown \{0\}$ which has $0$ as a
cluster point in the Bohr topology. Indeed, we can take $H=\tbigcup%
\limits_{m=1}^{\infty }H_{m}$ where for each positive integer $m$, $H_{m}$
is a finite subset of $(F-F)\diagdown \{0\}$ having the property that for
all $x_{1},...,x_{m}\in G$ there is some $\gamma \in H_{m}$ with 
\begin{equation*}
\sup_{1\leq j\leq m}\left\vert \gamma (x_{j})-1\right\vert <\frac{1}{m}\text{%
.}
\end{equation*}
\end{lemma}

\begin{proof}
To begin, we claim that there is some $\phi \in \overline{\Gamma }$ with the
property that if $V$ is an open neighbourhood of $\phi $, then $V\cap F$ is
infinite. If not, then for every $\phi \in \overline{\Gamma }$ there is an
open neighbourhood of $\phi $, say $U_{\phi }$, such that $U_{\phi }\cap F$
is finite. As the open sets $U_{\phi }$ cover $\overline{\Gamma }$ and $%
\overline{\Gamma }$ is compact, we can choose a finite subcover, $\{U_{\phi
_{j}}\}_{j=1}^{J}$. But that contradicts the assumption that $F$ is infinite
and hence proves the claim.

In particular, if $x_{1},...,x_{m}\in G$, then 
\begin{equation*}
V=\left\{ \gamma \in \overline{\Gamma }:\sup_{j=1,...,m}\left\vert \gamma
(x_{j})-\phi (x_{j})\right\vert <\frac{1}{2m}\right\}
\end{equation*}%
is a neighbourhood of $\phi $ and hence must contain infinitely many
elements from $F$. For each such collection $x_{1},...,x_{m}$, choose $%
f_{1}\neq f_{2}$ in $F$ such that $\left\vert f_{i}(x_{j})-\phi
(x_{j})\right\vert <1/(2m)$ for all $j=1,...,m$. Then $\left\vert
f_{1}(x_{j})-f_{2}(x_{j})\right\vert <1/m$ for $i=1,2$ and all $j=1,...,m$.
Since $f_{1}-f_{2}\in F-F\subseteq \Gamma ,$ and is therefore a continuous
character on $G$, there is a neighbourhood of $(x_{1},...,x_{m})\in G^{m}$,
denoted $W(f_{1},f_{2})$, such that if $(y_{1},...,y_{m})\in W(f_{1},f_{2})$%
, then 
\begin{equation*}
\sup_{j=1,...,m}\left\vert f_{1}(y_{j})-f_{2}(y_{j})\right\vert <\frac{1}{m}.
\end{equation*}%
As $G^{m}$ is compact, there are finitely many such neighbourhoods, $%
W(f_{1}^{(k)},f_{2}^{(k)})$, for $k=1,...,N_{m}$, so that $%
G^{m}=\tbigcup\limits_{k=1}^{N_{m}}W(f_{1}^{(k)},f_{2}^{(k)})$. The set 
\begin{equation*}
H_{m}=\{f_{2}^{(k)}-f_{1}^{(k)}:k=1,...,N_{m}\}
\end{equation*}
meets the requirements of the lemma.

The observation that $H=\tbigcup_{m=1}^{\infty }H_{m}$ clusters at $0$
follows directly from the fact that any neighbourhood of $0$ in the Bohr
topology contains a subset of the form $\{\gamma :\sup_{j=1,...,m}\left\vert
\gamma (x_{j})-1\right\vert <1/m\}$ for some positive integer $m$ and $%
x_{1},...,x_{m}\in G$.
\end{proof}

Before turning to the details of the proof of our main result, we list some
other elementary facts about Kronecker and $I_{0}$ sets which can be found
in \cite{GH13} and will be used in the proof of our main result.

\begin{lemma}
\label{tech}(i) Suppose $\pi :\Gamma \rightarrow \Lambda $ is a homomorphism
that is an injection on $E\subseteq \Gamma $. If $\pi (E)$ is weak $%
\varepsilon $-Kronecker (or $I_{0})$ as a subset of the group $\pi (\Gamma )$%
, then $E$ is weak $\varepsilon $-Kronecker (resp., $I_{0})$ as a subset of $%
\Gamma $.

(ii) Suppose $\Lambda $ is a subgroup of $\Gamma $ and that $E$ is weak $%
\varepsilon $-Kronecker (or $I_{0}$) as a subset of the group $\Lambda $.
Then $E$ is also weak $\varepsilon $-Kronecker (resp., $I_{0}$) as a subset
of $\Gamma $.
\end{lemma}

\section{The Main Result}

\subsection{Statement and outline of the proof}

Here is the statement of our main result.

\begin{theorem}
\label{main}If $F$ is any infinite subset of a discrete abelian group $%
\Gamma $, there are countable, disjoint sets $E\subseteq F$ and $E^{\prime
}\subseteq F+F-F$ such that both $E$ and $E^{\prime }$ are $I_{0}$, but $%
E\cup E^{\prime }$ is not $I_{0}$. Furthermore, the sets $E,E^{\prime }$ can
be chosen to be weak $\varepsilon $-Kronecker for suitable $\varepsilon \leq 
\sqrt{2}$.
\end{theorem}

\begin{remark}
The choice of $\varepsilon $ will be clear from the proof and depends on
algebraic properties of $F$. As will be seen in the proof, in many
situations $\varepsilon $ can be chosen to be arbitrarily small.
\end{remark}

Since the union of any two Sidon sets is again Sidon, we immediately obtain
the following corollary.

\begin{corollary}
Every infinite, discrete abelian group admits a Sidon set that is not $I_{0}$%
.
\end{corollary}

The remainder of the paper will be devoted to proving the theorem. Its proof
depends upon the general structure theory for abelian groups.

\begin{theorem}
(see \cite[p. 165]{GH13}) Given any discrete abelian group $\Gamma $, there
is an index set $\mathcal{I}$ such that $\Gamma $ is isomorphic to a
subgroup of 
\begin{equation*}
\Omega =\tbigoplus_{\alpha \in \mathcal{I}}\Omega _{\alpha },
\end{equation*}%
where for each $\alpha $ either $\Omega _{\alpha }$ $=\mathbb{Q}$ or there
is a prime number $p_{\alpha }$ such that $\Omega _{\alpha }=C(p_{\alpha
}^{\infty })$, the group of all $p_{\alpha }^{n}$-th roots of unity.
\end{theorem}

Throughout the remainder of the paper $\pi _{\alpha }$ will denote the
projection from $\Omega $ onto the factor group $\Omega _{\alpha }$. Our
proof of Theorem \ref{main} will be constructive and will depend on the
following two cases:

Case 1: There is some index $\alpha \in \mathcal{I}$ such that $\pi _{\alpha
}(F)$ is infinite. This case will be handled by Lemma \ref{L:caseQ} when $%
\Omega _{\alpha }=\mathbb{Q}$ and Lemma \ref{L:caseCPinfinity} when $\Omega
_{\alpha }=C(p_{\alpha }^{\infty })$. We will see that we can even arrange
for the sets $E$ and $E^{\prime }$ to be $\varepsilon $-Kronecker for any
given $\varepsilon >0$. Making the choice with $\varepsilon <\sqrt{2}$
ensures $E,$ $E^{\prime }$ are both $I_{0}$.

Case 2: Otherwise, $\pi _{\alpha }(F)$ is finite for all indices $\alpha \in 
\mathcal{I}$ and then there must be an infinite subset $\mathcal{J\subseteq I%
}$ such that for each $\alpha \in \mathcal{J}$ there is some $\lambda \in F$
with $\pi _{\alpha }(\lambda )\neq 0$. The existence of such an infinite
subset of indices allows us to either construct sets $E$, $E^{\prime }$ that
are weak $\varepsilon $-Kronecker for some $\varepsilon <\sqrt{2}$ (and
hence $I_{0})$ or to construct sets $E$, $E^{\prime }$ that are both
independent (and hence $I_{0})$ and weak $\sqrt{2}$-Kronecker. The choice of
construction depends on the orders of the non-zero characters $\pi _{\alpha
}(\lambda )$. This argument can be found in Lemma \ref{L:many factors}.

In both cases, the two sets we construct will be disjoint and have the form $%
E=\{\gamma _{n}\}$, $E^{\prime }=\{\gamma _{n}+\chi _{n}\}$ where $\{\chi
_{n}\}$ clusters at $0$. Thus the fact that $E\cup E^{\prime }$ is not $%
I_{0} $ will follow immediately from Lemma \ref{Mela}.

We now turn to handling these two cases.

\subsection{Proof of the Theorem in Case 1}

\begin{lemma}
\label{L:caseQ} Suppose there exists an index $\alpha \in \mathcal{I}$ such
that $\pi _{\alpha }(F)$ is infinite and $\Omega _{\alpha }=\mathbb{Q}$.
Given any $\varepsilon >0,$ there are infinite disjoint sets $E\subset F$
and $E^{\prime }\subset F+F-F$ such that $E$ and $E^{\prime }$ are weak $%
\varepsilon $-Kronecker and $I_{0}$, but $E\cup E^{\prime }$ is not $I_{0}$.
\end{lemma}

\begin{proof}
Let $H=\tbigcup_{m=1}^{\infty }H_{m}=\{\chi _{n}\}_{n=1}^{\infty }\subseteq
(F-F)\diagdown \{0\}$ be a set that clusters at $0,$ constructed in Lemma %
\ref{clustering}. Fix $0<\varepsilon <\sqrt{2}$ and assume we can find a
sequence of characters $\lambda _{n}\in \pi _{\alpha }(F)$ such that:

(a) $V=\{\lambda _{n}\}$ and $V^{\prime }=\{\lambda _{n}+\pi _{\alpha }(\chi
_{n})\}$ are weak $\varepsilon $-Kronecker sets in $\Omega _{\alpha };$ and

(b) For $n\neq n^{\prime }$ we have $\lambda _{n}\neq \lambda _{n^{\prime }}$%
, $\lambda _{n}\neq \lambda _{n^{\prime }}+\pi _{a}(\chi _{n^{\prime }})$,
and $\lambda _{n}+\pi _{\alpha }(\chi _{n})\neq \lambda _{n^{\prime }}+\pi
_{a}(\chi _{n^{\prime }})$.

Then, for each $\lambda _{n}$ choose some $\gamma _{n}\in F$ such that $\pi
_{\alpha }(\gamma _{n})=\lambda _{n}$. Set $E=\{\gamma _{n}\}\subseteq F$
and $E^{\prime }=\{\gamma _{n}+\chi _{n}\}\subseteq F+F-F$. By construction, 
$\pi _{\alpha }$ is one-to-one from $E$ to $V$ and one-to-one from $%
E^{\prime }$ to $V^{\prime }$. Condition (b) and the fact that $H$ consists
of nonzero characters implies that $E$ and $E^{\prime }$ consist of distinct
terms and are disjoint.

By Lemma \ref{tech}(i), $E$ and $E^{\prime }$ inherit the weak $\varepsilon $%
-Kronecker property from $V$ and $V^{\prime }$ respectively. Since $%
\varepsilon <\sqrt{2}$, both $E$ and $E^{\prime }$ are $I_{0}$. Furthermore,
because $0$ is a cluster point of the set $H$, Lemma \ref{Mela} implies that 
$E\cup E^{\prime }$ is not $I_{0}$.

Thus the proof of the lemma will be complete if we can construct a sequence
of characters satisfying the two conditions (a) and (b). This will be an
induction argument which depends on whether $\pi _{\alpha }(F$ $)$ is a
subset of a group isomorphic to $\mathbb{Z}$ or it is not.

First, suppose that there is some integer bound $B>0$ such that for all $%
\lambda \in \pi _{\alpha }(F)$, there are are integers $b>0$ and $a$ such
that $\pi _{\alpha }(\lambda )=a/b$ and $b\leq B$. Then $\pi _{\alpha }(F)$
is a subset of the (additive) subgroup $\frac{1}{B!}\mathbb{Z}$ of $\mathbb{Q%
}$. Because $H\subseteq F-F$, we also have $\pi _{\alpha }(H)\mathbb{\ }$%
contained in this subgroup.

Given $\varepsilon >0,$ choose an integer $q>2$ such that $\pi
/(q-1)<\varepsilon $. As $\pi _{\alpha }(F)$ is infinite, we may inductively
choose $\lambda _{n}\in \pi _{\alpha }(F),$ sufficiently large in modulus,
so that both $V$ and $V^{\prime }$ are Hadamard sequences in $\mathbb{Q}$
with Hadamard ratio $\geq q$ and condition (b) is satisfied$.$ By \cite[%
Prop. 2.2.6]{GH13}, $V$ and $V^{\prime }$ are both weak $\varepsilon $%
-Kronecker subsets of $\frac{1}{B!}\mathbb{Z}$ and by Lemma \ref{tech}(ii)
they are also both weak $\varepsilon $-Kronecker sets in $\Omega _{\alpha }$%
. Thus condition (a) is satisfied.

Otherwise, for every positive integer $B$, there is some $s/t\in \pi
_{\alpha }(F)$ with $t>B$ and $\gcd (s,t)=1$. (We will say $s/t$ is in
reduced form.) In this case we need to carefully account for the
denominators of rational numbers. Note that any $x\in \mathbb{Q}$ has a
unique reduced form, $s/t,$ and we will write $D(x)$ for the denominator $t$.

Given $\varepsilon >0,$ choose an integer $q>2$ such that $\pi
/q<\varepsilon $. Let $B_{0}=D(\pi _{\alpha }(\chi _{1}))$ and choose $%
\lambda _{1}\in \pi _{\alpha }(F)$ such that $D(\lambda _{1})>qB_{0}!$.
Assuming $\lambda _{1},...,\lambda _{n}$ have been inductively constructed
for $n\geq 1$, let 
\begin{equation*}
B_{n}=2\max \{D(\lambda _{i}),\,D(\pi _{\alpha }(\chi _{j})):\,1\leq i\leq
n,\,1\leq j\leq n+1\}.
\end{equation*}%
Now choose $\lambda _{n+1}\in $ $\pi _{\alpha }(F)$ so that $D(\lambda
_{n+1})>qB_{n}!$. This choice ensures that $\lambda _{i}$ and $\lambda
_{i}+\pi _{\alpha }(\chi _{i})$ for $1\leq i\leq n$, as well as $\pi
_{\alpha }(\chi _{n+1}),$ all belong to $\frac{1}{B_{n}!}\mathbb{Z}$, while $%
\lambda _{n+1}$ and $\lambda _{n+1}+\pi _{\alpha }(\chi _{n+1})$ are outside 
$\frac{1}{B_{n}!}\mathbb{Z}$. It follows that condition (b) will be
satisfied for $V$ and $V^{\prime }$.

We argue next that $V^{\prime }$ is $\varepsilon $-Kronecker in $\Omega
_{\alpha }$. To this end, let $\phi :V^{\prime }\rightarrow \mathbb{T}$, say 
$\phi (\lambda _{n}+\pi _{\alpha }(\chi _{n}))=t_{n}\in \mathbb{T}$. We need
to prove there is some character $g\in \widehat{\mathbb{Q}}$, the dual of $%
\mathbb{Q}$, such that 
\begin{equation}
\left\vert g(\lambda _{n}+\pi _{\alpha }(\chi _{n}))-t_{n}\right\vert
<\varepsilon \text{ for all \thinspace }n\text{.}  \label{charg}
\end{equation}%
As explained in \cite[25.5]{HewR}, elements of $\widehat{\mathbb{Q}}$ can be
identified with sequences $\{\omega _{n}\}\subset \mathbb{T},$ subject to
the constraints that $\omega _{n+1}^{n+1}=\omega _{n}$, with the
understanding that $g(1/n!)=\omega _{n}$. Clearly it will be sufficient to
satisfy the consistency condition 
\begin{equation}
\omega _{B_{n+1}}^{B_{n+1}!/B_{n}!}=\omega _{B_{n}},  \label{E:consistent}
\end{equation}%
provided that for $j\notin \{B_{n}\}$, say $B_{n}<j<B_{n+1},$ one specifies 
\begin{equation*}
\omega _{j}=\omega _{B_{n+1}}^{B_{n+1}!/j!}.
\end{equation*}

To start the specification of $g$, set $\omega _{B_{0}}=1$. This ensures
that if $k\in \mathbb{Z}$, then $g(k)=\omega _{B_{0}}^{kB_{0}!}$, hence $g(%
\mathbb{Z})=1$. (This will be helpful in the next lemma as it allows us to
interpret $g$ as a character on $\mathbb{Q}/\mathbb{Z}$.) Since $\pi
_{\alpha }(\chi _{1})=s/B_{0}!$ for some integer $s$, it follows that 
\begin{equation*}
g(\pi _{\alpha }(\chi _{1}))=\omega _{B_{0}}^{s}=1.
\end{equation*}%
By Equation (\ref{E:consistent}) one may choose $\omega _{B_{1}}$ to be any $%
J$-th root of unity, where $J=B_{1}!/B_{0}!$, in other words, one can choose
any integer $K\in \lbrack 0,J-1]$ and specify 
\begin{equation*}
\omega _{B_{1}}=e^{2\pi iK/J}.
\end{equation*}%
Because $\lambda _{1}\in \frac{1}{B_{1}!}\mathbb{Z}$, the reduced form of $%
\lambda _{1}$ is $s/t$ with $t$ dividing $B_{1}!$. Thus, with $y/z$ the
reduced form of $B_{0}!/t,$ we have 
\begin{equation*}
g(\lambda _{1})=\left( e^{2\pi iK/J}\right) ^{sB_{1}!/t}=e^{2\pi
iKsB_{0}!/t}=e^{2\pi iKsy/z}.
\end{equation*}%
Since $t=D(\lambda _{1})>qB_{0}!$, the reduced form of $B_{0}!/t$ is $y/z$
with $z>q$. Both $s$ and $y$ are relatively prime to $z$, hence the
exponential $e^{2\pi isy/z}$ is a primitive $z$-th root of unity. By
definition, $B_{1}\geq 2t$ and $t>qB_{0}!$, thus $J=B_{1}!/B_{0}!\geq
B_{1}>t\geq z$. This means we can choose any of the $z$-th roots of unity as
the value for $g(\lambda _{1})$. If we make a choice that is closest to $%
t_{1}$, then the angular difference between $g(\lambda _{1})$ and $t_{1}$ is
at most $\pi /z$ and thus 
\begin{equation*}
|g(\lambda _{1}+\pi _{\alpha }(\chi _{1}))-1|=|g(\lambda _{1})-t_{1}|<\frac{%
\pi }{z}<\frac{\pi }{q}<\varepsilon .
\end{equation*}

We proceed to define $\{\omega _{B_{n}}\}$ inductively. Assume that for $%
1\leq j\leq n$ we have specified $\omega _{B_{j}}$ so that $|g(\lambda
_{j}+\pi _{\alpha }(\chi _{j}))-t_{j}|<\varepsilon $. By the definition of $%
B_{n}$ we know that $\pi _{\alpha }(\chi _{n+1})$ is in $\frac{1}{B_{n}!}%
\mathbb{Z}$. Therefore $g$ has already been specified at $\pi _{\alpha
}(\chi _{n+1})$. However, $\lambda _{n+1}\notin $ $\frac{1}{B_{n}!}\mathbb{Z}%
,$ but is inside $\frac{1}{B_{n+1}!}\mathbb{Z}$, so the selection of $\omega
_{B_{n+1}}$ will determine $g(\lambda _{n+1})$.

By Equation (\ref{E:consistent}), $\omega _{B_{n+1}}$ can be chosen to be
any $J$-th root of $\omega _{B_{n}},$ where $J=B_{n+1}!/B_{n}!$. If we write 
$e^{i\theta }$ for $\omega _{B_{n}},$ then we are free to choose any integer 
$K\in \lbrack 0,J-1]$ and define 
\begin{equation*}
\omega _{B_{n+1}}=e^{i\theta /J}\cdot e^{2\pi iK/J}.
\end{equation*}%
Let $s/t$ be the reduced form of $\lambda _{n+1}$ and $y/z$ the reduced form
for $B_{n}!/t$. Since $t$ divides $B_{n+1}!,$ we have 
\begin{equation*}
\begin{aligned} g(\lambda_{n+1})&=\omega_{B_{n+1}}^{sB_{n+1}!/t}=e^{i \theta
sB_n!/t} \cdot e^{2 \pi i KsB_n!/t}=e^{i \theta sB_n!/t} \cdot e^{2 \pi i
Ksy/z}. \end{aligned}
\end{equation*}%
As before we see that any of the $z$-th roots of unity can be used to help
define $g(\lambda _{n+1})$ and we make the choice (of $K$) so that the
corresponding $z$-th root of unity differs in angle by at most $\pi /q$ from 
\begin{equation*}
\left( e^{i\theta sB_{n}!/t}\right) ^{-1}g(\pi _{\alpha }(\chi
_{n+1}))^{-1}t_{n+1}.
\end{equation*}%
(We remind the reader that the first two factors above are known as they
have already been determined by $\omega _{B_{n}}$.) Therefore 
\begin{equation*}
\begin{aligned} \left|g(\lambda_{n+1}+\pi_\alpha
(\chi_{n+1}))-t_{n+1}\right|&=\left|e^{2 \pi i Ksy/z}-\left(e^{i \theta
sB_n!/t}\right)^{-1} g(\pi_\alpha \chi_{n+1})^{-1} t_{n+1}\right|\\
&<\frac{\pi}{z} <\frac{\pi}{q} <\epsilon. \end{aligned}
\end{equation*}%
Since the choice of $g$ satisfies (\ref{charg}), it follows that $V^{\prime
} $ is $\varepsilon $-Kronecker.

The proof that $V$ is $\varepsilon $-Kronecker is similar, but easier, as
the factors $g(\pi _{\alpha }(\chi _{n+1}))$ are not present. This shows
that condition (a) also holds and that completes the proof of the Lemma.
\end{proof}

\begin{lemma}
\label{L:caseCPinfinity} Suppose there exists an index $\alpha \in \mathcal{I%
}$ such that $\pi _{\alpha }(F)$ is infinite and $\Omega _{\alpha
}=C(p^{\infty })$. Given any $\varepsilon >0,$ there are infinite disjoint
sets $E\subset F$ and $E^{\prime }\subset F+F-F$ such that $E,$ $E^{\prime }$
are weak $\varepsilon $-Kronecker and $I_{0}$, but $E\cup E^{\prime }$ is
not $I_{0}$.
\end{lemma}

\begin{proof}
Let $H=\{\chi _{n}\}\subseteq (F-F)\diagdown \{0\}$ be a countable set that
clusters at $0$, as in the previous lemma. We identify $C(p^{\infty })$ with
a subgroup of $\mathbb{Q}/\mathbb{Z},$ so that for every $\lambda $ in the
subgroup generated by $F$ there is some $x_{\lambda }\in \mathbb{Q}$ such
that $\pi _{\alpha }(\lambda )=x_{\lambda }+\mathbb{Z}$.

Because $\pi _{\alpha }(F)$ is infinite, the set of minimal denominators $%
\{D(x_{\lambda })\,:\,\lambda \in F\}$ must be unbounded. The proof of the
second part of Lemma \ref{L:caseQ} shows that given $0<\varepsilon <\sqrt{2}$
there is a sequence $V=\{x_{\lambda _{n}}\}$ such that both $V$ and $%
V^{\prime }=\{x_{\lambda _{n}}+x_{\chi _{n}}\}$ are $\varepsilon $-Kronecker
sets in $\mathbb{Q},$ with the interpolation being done by characters $g\in 
\widehat{\mathbb{Q}}$ such that $g(\mathbb{Z})=1$, and hence by characters
on $\mathbb{Q}/\mathbb{Z}$. These can also be viewed as characters on $%
C(p^{\infty })$ if the domain is suitably restricted. Of course, $g(\pi
_{\alpha }(\lambda _{n}))=g(x_{\lambda _{n}}+\mathbb{Z})=g(x_{\lambda _{n}})$
and $g(\pi _{\alpha }(\lambda _{n}+\chi _{n}))=g(x_{\lambda _{n}}+x_{\chi
_{n}})$. It follows that both $\{\pi _{\alpha }(\lambda _{n})\}$ and $\{\pi
_{\alpha }(\lambda _{n}+\chi _{n})\}$ are $\varepsilon $-Kronecker subsets
of $\Omega _{\alpha }$.

The construction in the proof of the previous lemma also ensures that the
pullbacks, $E=\{\lambda _{n}\}$ and $E^{\prime }=\{\lambda _{n}+\chi _{n}\},$
are disjoint, have distinct terms and are $\varepsilon $-Kronecker in $%
\Gamma .$ That their union is not $I_{0}$ follows immediately from Lemma \ref%
{Mela}.
\end{proof}

\begin{remark}
We note that similar arguments can be used to prove that under the
assumption that $\pi _{\alpha }(F)$ is infinite for some $\alpha $ there are
infinite sets, $E,E^{\prime }$, that are $I_{0}$ and have the property that
for each $\varepsilon >0$ the sets $E$ and $E^{\prime }$ contain cofinite
subsets that are $\varepsilon $-Kronecker and whose union is not $I_{0}$.
The latter statement is a consequence of the fact that the union of an $%
I_{0} $ set and a finite set is known to be $I_{0}$ (see \cite[p.63]{GH13}).
\end{remark}

\subsection{Proof of the Theorem in Case 2}

\begin{lemma}
\label{L:many factors} Suppose $\pi _{\alpha }(F)$ is finite for all $\alpha
\in \mathcal{I}$, but that $\mathcal{I}_{q}$ is infinite for some $q\geq 2$,
where 
\begin{equation*}
\mathcal{I}_{q}=\{\alpha \in \mathcal{I}:\exists \lambda \in F\text{ s.t. $%
\pi _{\alpha }(\lambda )$ has order at least $q$}\}.
\end{equation*}%
Let $|\exp (i\pi /q)-1|=\varepsilon _{q}$. There are infinite disjoint sets $%
E\subset F$ and $E^{\prime }\subset F+F-F$ such that $E$ and $E^{\prime }$
are both weak $\varepsilon _{q}$-Kronecker and $I_{0}$, but $E\cup E^{\prime
}$ is not $I_{0}$.
\end{lemma}

\begin{proof}
Again, let $H=\{\chi _{n}\}_{n=1}^{\infty }\subseteq (F-F)\diagdown \{0\}$
be a countable set that clusters at $0$.

First, assume $\mathcal{I}_{q}$ is infinite for some $q\geq 3$. Let $\beta
_{1}\in \mathcal{I}$ be chosen with the property that $\pi _{\beta
_{1}}(\chi _{1})=0$ and there is some $\lambda _{1}\in F$ with $\pi _{\beta
_{1}}(\lambda _{1})$ having order $\geq q$. We can do this since there are
only finitely many indices $\beta $ with $\pi _{\beta }(\chi _{1})\neq 0$.
Now inductively choose $\beta _{n}\in \mathcal{I}$ and $\lambda _{n}\in F$
such that $\pi _{\beta _{n}}(\lambda _{n})$ has order at least $q$, 
\begin{equation*}
\pi _{\beta _{n}}(\chi _{m})=0\text{ for all }m\leq n\text{ and }\pi _{\beta
_{n}}(\lambda _{m})=0\text{ for all }m<n\text{.}
\end{equation*}%
Set $\Pi $ equal to the projection from $\Omega $ onto $\Lambda =\bigoplus
\Omega _{\beta _{n}}$. Then $\{\Pi (\lambda _{n})\}$ and $\{\Pi (\lambda
_{n}+\chi _{n})\}$ are sequences of distinct elements of $\Lambda $ such
that $\Pi (\lambda _{m})\neq \Pi (\lambda _{n}+\chi _{n})$ if $m\neq n$.
Since $\chi _{n}\neq 0$, the sets $E=\{\lambda _{n}\}$ and $E^{\prime
}=\{\lambda _{n}+\chi _{n}\}$ are disjoint. Moreover $\Pi (\lambda _{n}+\chi
_{n})=\pi _{\beta _{n}}(\lambda _{n})+\rho _{n}$ where $\pi _{\beta
_{k}}(\rho _{n})=0$ for all $k\geq n$.

We claim that $\{\Pi (\lambda _{n}+\chi _{n})\}$ is weak $\varepsilon _{q}$%
-Kronecker. To prove this, note that a character $g$ on $\Lambda $ is
specified as $g=\{g_{n}\}$, with each $g_{n}$ a character on $\Omega _{\beta
_{n}}$. Let $M_{n}$ be the order of $\pi _{\beta _{n}}(\lambda _{n})$. For
any $M_{n}$-th root of unity, $\omega _{n},$ there is a character $g_{n}$
such that $g_{n}(\pi _{\beta _{n}}(\lambda _{n}))=\omega _{n}$. Thus given $%
\{t_{n}\}\subseteq \mathbb{T}$, we may inductively specify $g_{n}$ so that 
\begin{equation*}
|g(\pi _{\beta _{n}}(\lambda _{n}))-g(\rho _{n})^{-1}t_{n}|\leq |\exp (i\pi
/M_{n})-1|\leq \varepsilon _{q}.
\end{equation*}%
Thus%
\begin{equation*}
\left\vert g(\Pi (\lambda _{n}+\chi _{n}))-t_{n}\right\vert =|g(\pi _{\beta
_{n}}(\lambda _{n}))g(\rho _{n})-t_{n}|\leq \varepsilon _{q},
\end{equation*}%
which proves the claim.

The argument that $\{\Pi (\lambda _{n})\}$ is weak $\varepsilon _{q}$%
-Kronecker is similar. By Lemma \ref{tech}, $E$ and $E^{\prime }$ are also
weak $\varepsilon _{q}$-Kronecker. As $q\geq 3,$ we have $\varepsilon
_{q}\leq 1$ and hence $E$ and $E^{\prime }$ are $I_{0}$. As before, their
union is not $I_{0}$.

Otherwise, we can assume $\mathcal{I}_{2},$ but not $\mathcal{I}_{3},$ is
infinite. Then there is a finite (possibly empty) set $\mathcal{J\subseteq I}
$ such that for all $\lambda \in F$ and all $\alpha \in \mathcal{I\diagdown J%
}$, either $\pi _{\alpha }(\lambda )=0$ or has order $2$. Repeat the
construction as above, but this time with the additional requirement that $%
\beta _{n}\in \mathcal{I\diagdown J}.$ As before, the sets $E,$ $E^{\prime }$
that arise from the construction are disjoint and they are both weak $\sqrt{2%
}$-Kronecker. Moreover, $\Pi (E)$ and $\Pi (E^{\prime })$ are both
independent sets of elements of order $2$ and one can easily verify that
such sets are $I_{0}$ (c.f. \cite[p. 66]{GH13}). It follows from Lemma \ref%
{tech} that $E$ and $E^{\prime }$ are both $I_{0}$ sets, while their union
is not.
\end{proof}

These three lemmas complete the proof of the main theorem since the
assumption that $F$ is infinite guarantees that either $\pi _{\alpha }(F)$
is infinite for some $\alpha $, or there are infinitely many indices $\alpha 
$ with $\pi _{\alpha }(F)$ not trivial and in that case $\mathcal{I}_{q}$ is
infinite for some $q\geq 2$.

\end{document}